\documentclass[11pt]{article}

\usepackage[utf8]{inputenc}
\usepackage[english]{babel}
\usepackage{amsmath,amsthm,amsfonts}
\usepackage[font=small]{caption}
\usepackage{float}
\usepackage{geometry}
\usepackage{graphicx}
\usepackage{hyperref}
\usepackage{indentfirst}
\usepackage{wrapfig}

\hypersetup{colorlinks}

\newtheorem{theorem}{Theorem}
\newtheorem{lemma}{Lemma}

\title{Two lower bounds for maximum matching}

\author{Fedor Kuyanov\footnote{feodor.kuyanov@gmail.com; HSE University, Moscow, Russian Federation}}

\date{}

\begin{document}

\maketitle

\begin{abstract}
	In this paper, we prove two lower bounds for the maximum matching size in an arbitrary undirected graph. Despite their simplicity, these results are not widely known. This article aims to bring pleasure to the reader by giving short combinatorial proofs of these easy-to-state estimates.
\end{abstract}

\section{Main results}

We focus on simple undirected graphs with $n$ vertices, $m$ edges, and having the maximum degree $d$. We investigate the following question: how small can its maximum matching be? Surprisingly, the answer to this question is hard to find in the literature. In the paper by P. Erdős and T. Gallai \cite{Erdos-Gallai}, the authors present a lower bound (Theorem 4.1); however, their proof seems quite complicated. Another related work is the article by V. Chvátal and D. Hanson \cite{Chvatal-Hanson} providing a sharp estimate for the maximum number of edges in a graph not containing a matching of a specific size. Their estimate involves three cases and is not particularly easy to state. In our work, we present two easy-to-state estimates and give short combinatorial proofs. Note that these estimates follow immediately from the much harder Vizing's theorem \cite{Vizing}. 

Let $G = (V, E)$ be a simple undirected graph (i.e. without loops and multiple edges). Denote $n := |V|$, $m := |E|$, and $d := \max_{v \in V} \deg(v)$. A subset of edges is called a \textit{matching} if no two edges have a common vertex. Let $k$ be the maximum matching size. We prove two following lower bounds:

\begin{enumerate}
	\item $k \ge m / n$ (Theorem \ref{estimate1})
	\item $k \ge 2 / 3 \cdot m / d$ (Theorem \ref{estimate2})
\end{enumerate}

Note that these results strengthen the trivial lower bounds of $m / (2n)$ and $m / (2d)$ by a constant factor. Also, the constant $2 / 3$ in the second estimate is sharp: consider a graph consisting of $r$ disjoint triangles; its maximum matching has $r = 2 / 3 \cdot 3r / 2$ edges.

Let $M$ be the set of vertices covered by the maximum matching $\pi_{max}$. For $u \in M$, denote by $d_{in}(u)$ and $d_{out}(u)$ the number of the incident edges in $G$ that connect $u$ with vertices inside and outside $M$, respectively. First, let us present the following observations.

\begin{lemma}
\label{degree-lemma}
	For each $(u, v) \in \pi_{max}$ we have
	\begin{enumerate}
		\item $d_{in}(u) + d_{out}(u) \le d$,
		\item $d_{in}(u) \le 2k$,
		\item $d_{out}(u) \le n - 2k$,
		\item either $d_{out}(u) = d_{out}(v) = 1$ or at least one of $\{d_{out}(u), d_{out}(v)\}$ is zero.
	\end{enumerate}
\end{lemma}

\begin{proof}
	The first claim is trivial because $d_{in}(u) + d_{out}(u) = \deg(u)$. The second and third claims are also clear because the graph has no multiple edges and thus $d_{in}(u) \le |M| = 2k$ and $d_{out}(u) \le |V \setminus M| = n - 2k$. Let us prove the last claim.
	
	Consider the sets of vertices $N(u), N(v) \subset V \setminus M$ which are adjacent to $u$ and $v$, respectively. We have $d_{out}(u) = |N(u)|$ and $d_{out}(v) = |N(v)|$. If at least one of them is empty, then all is done. Otherwise, assume that $|N(v)| \ge 2$. Take $u' \in N(u)$ and $v' \in N(v)$ such that $u' \neq v'$. In this case, we can improve the matching $\pi_{max}$ by removing the edge $(u, v)$ and adding $(u, u')$ and $(v, v')$ instead. This contradicts the maximality of $\pi_{max}$.
\end{proof}

\begin{lemma}
\label{edge-lemma}
	We have $m = \sum_{(u, v) \in \pi_{max}} \left[\frac{d_{in}(u) + d_{in}(v)}{2} + d_{out}(u) + d_{out}(v)\right]$.
\end{lemma}

\begin{proof}
	Note that all edges of $G$ are either inside $M$ or go from $M$ to $V \setminus M$; there are no edges completely outside $M$ because one could add them to $\pi_{max}$. Let $m_{in}$ and $m_{out}$ be the number of those edges. By a simple counting argument,
	\begin{align*}
		2m_{in} = \sum_{(u, v) \in \pi_{max}} d_{in}(u) + d_{in}(v) \qquad\text{and}\qquad m_{out} = \sum_{(u, v) \in \pi_{max}} d_{out}(u) + d_{out}(v),
	\end{align*}
	from which the statement of the lemma follows.
\end{proof}

Now we are ready to prove two lower bounds for the size of the maximum matching.

\begin{theorem}
\label{estimate1}
	For each graph G we have $k \ge m / n$.
\end{theorem}

\begin{proof}
	We have the following two cases:
	\begin{enumerate}
		\item $n - 2k \le 1$. In this case, $k \ge (n - 1) / 2 \ge m / n$, where the last inequality follows from the estimate $m \le n(n - 1) / 2$.
		\item $n - 2k \ge 2$. In this case, for $(u, v) \in \pi_{max}$, from the last two claims of Lemma \ref{degree-lemma} it follows that $d_{out}(u) + d_{out}(v) \le n - 2k$. Substituting this inequality into Lemma \ref{edge-lemma} and applying the second claim of Lemma \ref{degree-lemma}, we get
			\begin{align*}
				m \le \sum_{(u, v) \in \pi_{max}} \left[\frac{4k}{2} + n - 2k\right] = kn,
			\end{align*}
			proving the desired estimate.
	\end{enumerate}
\end{proof}

\begin{theorem}
\label{estimate2}
	For each graph G with the maximum degree $d$ we have $k \ge 2 / 3 \cdot m / d$.
\end{theorem}

\begin{proof}
	If $d = 1$, then $G$ is a matching and thus $k = m \ge \frac{2m}{3}$. Hereafter we assume $d \ge 2$. For each $(u, v) \in \pi_{max}$ we will show that $s(u, v) := \frac{d_{in}(u) + d_{in}(v)}{2} + d_{out}(u) + d_{out}(v) \le \frac{3}{2} d$. From the last claim of Lemma \ref{degree-lemma} we have the following two cases:
	\begin{enumerate}
		\item $d_{out}(u) = d_{out}(v) = 1$. From the first claim of Lemma \ref{degree-lemma} it follows that $d_{in}(u), d_{in}(v) \le d - 1$. Thus,
			\begin{align*}
				s(u, v) \le d + 1 \le \frac{3d}{2}.
			\end{align*}
		\item $d_{out}(v) = 0$ (wlog). Again, by the first claim of Lemma \ref{degree-lemma},
			\begin{align*}
				s(u, v) = \left(\frac{d_{in}(u)}{2} + d_{out}(u)\right) + \frac{d_{in}(v)}{2} \le (d_{in}(u) + d_{out}(u)) + d_{in}(v) / 2 \le d + d / 2 = \frac{3d}{2}.
			\end{align*}
	\end{enumerate}
	Finally, plugging this inequality into Lemma \ref{edge-lemma} we have
	\begin{align*}
		m = \sum_{(u, v) \in \pi_{max}} s(u, v) \le \frac{3kd}{2} \implies k \ge 2 / 3 \cdot m / d.
	\end{align*}
\end{proof}

\section{Acknowledgements}

This paper was prepared within the framework of the HSE University Basic Research Program. We thank Benjamin Sudakov, Mikhail Vyalyi, Vladimir Podolskii, and Stasys Jukna for useful discussions on this topic. We also thank Martin Milanič for providing the bibliographical references.

\end{document}